\documentclass[11pt]{amsart}
\usepackage[utf8]{inputenc}

\usepackage[nospace,noadjust]{cite}
\usepackage{relsize}
\usepackage{cite}
\usepackage{color}
\usepackage[dvipsnames]{xcolor}

\usepackage{tikz-cd}
\usetikzlibrary{calc,fit,matrix,arrows,automata,positioning}
\usetikzlibrary{decorations.pathreplacing,angles,quotes}
\usepackage{colortbl}
\usepackage{hyperref, enumitem}
\hypersetup{
  colorlinks   = true, 
  urlcolor     = blue, 
  linkcolor    = Purple, 
  citecolor   = red 
}
\usepackage{amsmath,amsthm,amssymb}

\usepackage[margin=2.5cm]{geometry}

\usepackage{stmaryrd}

\usepackage{array}
\newcolumntype{x}[1]{>{\centering\arraybackslash\hspace{0pt}}p{#1}}

\theoremstyle{definition}
\newtheorem{theorem}{Theorem}[section]
\newtheorem{definition}[theorem]{{{Definition}}}
\newtheorem{example}[theorem]{{{Example}}}
\newtheorem{notation}[theorem]{{{Notation}}}
\newtheorem{remark}[theorem]{{{Remark}}}

\newtheorem{corollary}[theorem]{{{Corollary}}}
\newtheorem{proposition}[theorem]{{{Proposition}}}
\newtheorem{lemma}[theorem]{{{Lemma}}}

\newtheorem{construction}{{{Construction}}}

\newcommand{\numberset}{\mathbb}
\newcommand{\N}{\numberset{N}}

\newcommand{\F}{\numberset{F}}

\newcommand{\mS}{\mathcal{S}}

\newcommand{\mC}{\mathcal{C}}
\newcommand{\mP}{\mathcal{P}}

\newcommand{\mL}{\mathcal{L}}

\newcommand{\mI}{\mathcal{I}}

\newcommand{\mB}{\mathcal{B}}
\newcommand{\mF}{\mathcal{F}}

\newcommand{\mD}{\mathcal{D}}

\newcommand{\sH}{\sigma}

\newcommand{\mV}{\mathcal{V}}
\newcommand{\mR}{\mathcal{R}}

\newcommand{\mH}{\mathcal{H}}

\newcommand{\wt}{\textnormal{wt}}

\newcommand{\Fq}{\F_q}

\newcommand{\rk}{\textnormal{rk}}

\DeclareMathOperator{\GL}{GL}

\DeclareMathOperator{\PG}{PG}
\DeclareMathOperator{\PGL}{PGL}
\DeclareMathOperator{\rs}{rowsp}

\newcommand{\Fhk}{[n,k]_{q^h/q}}

\newcommand{\St}{\,:\,}

\title{Outer Strong Blocking Sets}
\usepackage[foot]{amsaddr}
\author{Gianira N. Alfarano$^1$}
\author{Martino Borello$^{2,3}$}
\author{Alessandro Neri$^4$}
\address{$^1$Eindhoven University of Technology, the Netherlands.}
\address{$^2$Universit\'e Paris 8, Laboratoire de G\'eom\'etrie, Analyse et Applications, LAGA, Universit\'e Sorbonne Paris Nord, CNRS, UMR 7539, France.}
\address{$^3$INRIA Saclay \& LIX, CNRS UMR 7161, École Polytechnique, France.}
\address{$^4$Ghent University, Belgium.}

\email{gianira.alfarano@gmail.com, martino.borello@univ-paris8.fr, alessandro.neri@ugent.be}

\begin{document}

\maketitle

\begin{abstract} 
Strong blocking sets, introduced first in 2011 in connection with saturating sets, have recently gained a lot of attention due to their correspondence with minimal codes. In this paper, we dig into the geometry of the concatenation method, introducing the concept of outer strong blocking sets and their coding theoretical counterpart. We investigate their structure and provide bounds on their size. As a byproduct, we improve the best-known upper bound on the minimum size of a strong blocking set. Finally, we present a geometric construction of small strong blocking sets, whose computational cost is significantly smaller than the previously known ones.
\end{abstract}


\section*{Introduction}
A \emph{strong blocking set} is a set $\mP$ of points in the projective space $\PG(k-1,q)$, such that for every hyperplane $H\subseteq \PG(k-1,q)$ the intersection $\mP\cap H$ spans $H$. These combinatorial objects, introduced in \cite{davydov2011linear} in connection with saturating sets, have been recently investigated, not only for their theoretical properties but also for their use in applications, in particular in coding theory and cryptography. 

Independently in \cite{alfarano2019geometric} and \cite{tang2019full}, it was initially shown that projective equivalence classes of strong blocking sets in $\PG(k-1,q)$ are in one-to-one correspondence with monomial equivalence classes of projective \emph{minimal codes} of dimension $k$ over $\F_q$. A code is minimal if all its nonzero codewords are minimal, i.e. their support does not contain the support of any other linearly independent codeword. First results on minimal codes
were presented in \cite{chabanne2013towards}, in connection to secure two-party computation. Bounds on the parameters of a minimal code were investigated in  \cite{alfarano2019geometric, alfarano2022three, chabanne2013towards,cohen2013minimal,lu2021parameters}. In \cite{ashikhmin1998minimal}, a sufficient
condition (now called the \emph{Ashikhmin-Barg condition}) for a code to be minimal was presented and later has been extensively used to get constructions. 
Only recently, minimal codes have been studied through their connection with strong blocking sets; see \cite{alfarano2019geometric,alfarano2022three,lu2021parameters,tang2019full}. 

A remarkable property of minimal codes is that they are asymptotically good; see \cite{alfarano2019geometric,cohen2013minimal}. In this regard, one of the main research directions of the last years focused on finding explicit constructions of minimal codes whose length grows linearly with their dimension, or equivalently, providing constructions of infinite families of strong blocking sets whose size grows linearly with the dimension of the ambient projective space.
Strong blocking sets realized as union of subspaces are the most common constructions. Indeed, if the pointset in $\PG(k-1,q)$ consists of union of subspaces it is usually easier to control the intersection with hyperplanes. In particular, many of the known constructions of strong blocking sets are union of lines; see \cite{davydov2011linear,fancsali2014lines,alfarano2019geometric,alfarano2022three,bartoli2022cutting,bartoli2021small}. 
In \cite{heger2021short}, an upper bound on the size of the smallest strong blocking set in $\PG(k-1,q)$ is provided using probabilistic methods, which improves upon the previously known bounds. In the recent paper \cite{bartoli2021small}, the authors explicitly construct families of asymptotically good minimal codes via the method of \emph{concatenation}, already exploited in \cite{cohen2013minimal}. In this way, they also
obtain explicit constructions of small strong blocking sets in projective spaces, whose size linearly depends on the dimension of the ambient space. 
In order to do so, they leverage explicit constructions of asymptotically good algebraic-geometry (AG) codes.  Moreover, an analogue of the Ashikhmin-Barg condition in the context of concatenated codes, denominated \emph{Outer AB condition} is provided. 

In this paper we pursue the study of strong blocking sets in connection with minimal codes. We provide a deeper geometric insight on the concatenation method, introducing the notions of outer strong blocking sets and outer minimal codes, which include the classical ones. These are sets whose field reduction is a strong blocking set and codes whose concatenation with minimal codes is minimal.
Since adding a point to an outer strong blocking set gives a larger outer strong blocking set, we are interested in the smallest sets with this property. We identify a sufficient property that a collection of subsets of points in $\PG(k-1,q)$ can satisfy for their union to form a strong blocking set. This turns out to be the basis for generalizing different results from \cite{fancsali2014lines} and \cite{heger2021short}. In particular, we provide lower and upper bounds on the smallest size of an outer strong blocking sets and as a byproduct, we get the best-known general upper bound on the cardinality $m(k,q)$ of the smallest strong blocking set in $\PG(k-1,q)$: 
\[m(k,q)\leq \left\lceil \frac{1}{\log_{q^2}\left(\frac{q^4}{q^3-q+1}\right)}\cdot k\right\rceil\cdot (q+1),\]
for $k$ even. For odd dimensions, we get essentially the same bound by projection. We also highlight the impact of this new result on saturating sets. Finally, we provide an explicit construction of small strong blocking sets obtained by iterative field reduction. The computational cost of our construction is significantly smaller than the ones proposed in \cite{bartoli2021small,alon2023strong}, since these require to construct first an AG code and the fastest algorithm for constructing such a code has a much higher complexity; see Section \ref{sec:constructions}. 
 \bigskip
\paragraph{\textbf{Outline}} The paper is organized as follows. In Section \ref{sec:preliminaries} we provide the necessary background on strong blocking sets and minimal codes. In Section \ref{sec:fieldreduction} we explain the concatenation method and establish its connection to the field reduction concept in geometry. Section \ref{sec:OSBS} is devoted to the introduction  and the study of outer strong blocking sets and outer minimal codes. In Section \ref{sec:bounds} we present a lower bound on the size of outer strong blocking set and the best-known general upper bound on the size of the smallest strong blocking set for a given dimension and given field size. Moreover, a connection with saturating sets is highlighted. In Section~\ref{sec:constructions} we provide an iterative construction of small outer strong blocking sets and we study its computational complexity.
\bigskip
\paragraph{\textbf{Notation}}
Throughout this paper, $q$ denotes a prime power, $\F_q$ is the finite field with $q$ elements, and $h$, $n$, $k$ are integers with $1 \le k \le n$ and $h\geq 1$. For $i \in \N =\{0,1,2, \ldots\}$ we let $[i]:=\{j \in \N \, : \, 1 \le j \le i\}$. We denote by $\PG(k-1,q)$ the $(k-1)$-dimensional projective space over $\F_q$ and by $\mathrm{AG}(k-1,q)$ the $(k-1)$-dimensional affine space over~$\F_q$.
For any matrix $M\in \F_q^{a \times b}$ we denote by $\mathrm{rowsp}(M)$ and $\mathrm{colsp}(M)$ the rowspace and the columnspace of $M$ over $\F_q$, respectively, that is the $\F_q$-subspace of $\F_q^b$ generated by the rows of $M$ and the $\F_q$-subspace of $\F_q^a$ generated by the columns of $M$. 
Given a subset $S$ of an $\F_q$-vector space~$V$, we denote by $\langle S \rangle_{\F_q}$ its $\F_q$-span, that is the $\F_q$-vector subspace generated by all the elements of $S$. Given instead a subset $T$ of a projective space $\PG(k-1,q)$, we denote by $\langle T \rangle$ its projective span, that is the projective subspace generated by all the elements of $T$. For a vector $v\in\F_q^k$, we denote by $v^\perp$ the $(k-1)$-dimensional subspace of~$\F_q^k$ orthogonal to $\langle v\rangle_{\F_q}$, with respect to the standard inner product.

\section{Preliminaries}\label{sec:preliminaries}
In this section we briefly recall the notions of strong blocking sets and linear codes, with a particular focus on minimal codes.

\subsection{Strong Blocking Sets}

The concept of \emph{strong blocking set} has been introduced in \cite{davydov2011linear} in order to construct saturating sets in projective spaces over finite fields. In \cite{fancsali2014lines} strong blocking sets are referred to
as \emph{generator sets} and they are constructed as union of disjoint lines. In \cite{bonini2021minimal} they were re-introduced under the name of \emph{cutting blocking sets} in order to construct a family of minimal codes. In this subsection, we recall some properties of strong blocking sets, existence conditions and known results about their size.
 
 \begin{definition}
A set of points $\mP$ in $\PG(k-1,q)$ is called \textbf{strong blocking set} if for every hyperplane $H\subseteq \PG(k-1,q)$ we have that
$$ \langle \mP \cap H\rangle =H.$$
\end{definition}

\begin{example}\label{ex1}
Clearly, the whole projective space $\PG(k-1,q)$ is a strong blocking set.
\end{example}

Observe that from a given strong blocking set we can always obtain a larger one by adding any set of points and considering multiplicities, hence, with a little abuse of notation, we can also consider multisets of points as strong blocking sets. Moreover, for the same reason, it is interesting to know how small a strong blocking set can be. A lower bound was shown in \cite{alfarano2022three}, which is, however, not tight.

\begin{theorem}[{\cite[Theorem 2.14]{alfarano2022three}}] The size of a strong blocking set in $\PG(k - 1, q)$ is at least $(q + 1)(k - 1)$.
\end{theorem}

One way to construct small strong blocking sets is by taking unions of lines. We are going to recall two constructions of small strong blocking sets given in this way. The first one was proposed by Fancsali and Sziklai in \cite{fancsali2014lines} and it works as follows. Choose  $2k-3$ distinct points on the rational normal curve in $\PG(k - 1, q)$ and take the union of the tangent lines at these points. The resulting set is a strong blocking set, under the assumption that $q\geq 2k-3$ and $\mathrm{char}(\Fq)\geq k$. The latter condition can be removed, by using the method of \emph{diverted tangents}, also described in \cite{fancsali2014lines}. We refer to a strong blocking set of this form as a \textbf{rational normal tangent set}. The drawback of this construction is the constraint on $q$. In particular, it cannot be used for a fixed field with large values of $k$. Smaller constructions have been provided for $k=3,4$ (see \cite{fancsali2014lines}), $k=5$ (see \cite{bartoli2020resolving}) and $k=6$ (see \cite{bartoli2022cutting}). 
The second construction, that we will call \textbf{tetrahedron}, works for every choice of parameters $k,q$ and it was provided independently in \cite{davydov2011linear, alfarano2019geometric, bartoli2021inductive, lu2021parameters}.  Consider $k$ points $P_1,\ldots, P_k$ in general position in $\PG(k - 1, q)$.  Then, the union of all  lines passing through each pair of distinct points $\{P_i,P_j\}$ forms a strong blocking set.

Denote by $m(k,q)$ the size of the smallest strong blocking sets in $\PG(k - 1, q)$.
In \cite{heger2021short} a probabilistic argument gives an upper bound on $m(k,q)$, which refines the result in \cite{chabanne2013towards}. 

\begin{theorem}[\cite{heger2021short}] We have that 
$$m(k,q)\leq \begin{cases}
  \frac{2k-1}{\log_2\left(\frac{4}{3}\right)}  & \textnormal{ if } q=2, \\
  (q+1) \cdot\left\lceil  \frac{2}{1+\frac{1}{(q+1)^2\ln q}}   \cdot(k-1)  \right\rceil & \textnormal{ otherwise. }
\end{cases}$$
\end{theorem}

\medskip

\subsection{Linear Codes and Projective Systems}

Let $v\in\F_q^n$. The \textbf{support} of $v$ is defined as 
$\sH(v)=\{i \mid v_i \neq 0\} \subseteq [n]$
and its (Hamming)-\textbf{weight} is $\wt(v)=|\sH(v)|$.

An $[n,k]_q$ (linear) \textbf{code} is an $\F_q$-linear subspace $\mC \subseteq \F_q^n$ of dimension $k$. The vectors in $\mC$ are called \textbf{codewords}. The \textbf{minimum (Hamming) distance} of $\mC$ is defined as
$${\rm d}(\mC)=\min\{\wt(c) \mid c \in \mC, \, c \neq 0\}.$$  If $d={\rm d}(\mC)$ is known, we say that $\mC$ is 
an $[n,k,d]_q$ code. A \textbf{generator matrix}  $G\in\F_q^{k\times n}$ of $\mC$ is a matrix such that $\mathrm{rowsp}(G)=\mC$. We say that $\mC$ is \textbf{nondegenerate} if there is no $i~\in~[n]$ such that $c_i=0$ for all $c\in \mC$, and that $\mC$ is \textbf{projective} if in one (and thus in all) generator matrix of $\mC$ no two columns are $\Fq$-linearly dependent. 
Codes $\mathcal{C}$ and $\mathcal{C}^\prime$ are (\textbf{monomially})  \textbf{equivalent} if there exists an $\F_q$-linear isometry $f: \F_q^n \to \F_q^n$ with $f(\mC)=\mC'$.

\begin{definition}\label{def:projsystem}
 A \textbf{projective} $[n,k]_q$ \textbf{system} $\mathcal{P}$ is a finite  set of $n$ points (counted with multiplicity) of $\PG(k-1,q)$ that do not all lie on a hyperplane.
 Projective $[n,k]_q$ systems $\mathcal{P}$ and $\mathcal{P}^\prime$ are \textbf{equivalent} if there exists $\phi \in \mathrm{PGL}(k,q)$ mapping $\mathcal{P}$ to $\mathcal{P}^\prime$ which preserves the multiplicities of the points.
\end{definition}

There is a well-known correspondence\label{page:correspondence} between the (monomial) equivalence classes of nondegenerate $[n,k]_q$ linear codes and the equivalence classes of projective $[n,k]_q$ systems; see \cite[Theorem~1.1.6]{MR1186841}.
More precisely, let $G$ be a $k\times n$ generator matrix of an $[n,k]_q$ linear code. Consider the set $\mathcal{P}$ of one-dimensional subspaces of $\F_q^k$ spanned by the columns of~$G$, which gives a set of points in $\PG(k-1,q)$.  Conversely, let $\mathcal{P}$ be a projective $[n,k]_q$ system. Choose a representative for any point of $\mathcal{P}$ and consider the code generated by the matrix having these representatives as columns.
 
\begin{notation}\label{not:PC}
Let $\mC$ be an $[n,k]_q$ code. Since all the results that we provide are independent from the choice of the generator matrix of $\mC$, with a little abuse of notation, in the sequel we denote by $\mP(\mC)$ an arbitrary projective system in the equivalence class corresponding to the class of $\mC$.
\end{notation}

\begin{definition} \label{def:mc}
Let $\mC$ be an $[n,k]_q$ code.
A nonzero codeword $c \in \mC$ is called \textbf{minimal} if its support does not contain the support of any other linearly independent codeword, i.e. for  every codeword $c^\prime \in \mC$ with $\sigma(c^\prime) \subseteq \sigma(c)$, there exists $\lambda\in\F_q$ such that $c^\prime=\lambda c$. We say that $\mC$ is \textbf{minimal} if all its nonzero codewords are minimal.
\end{definition}

\begin{example}\label{ex2}
Let $G$ be the $k\times (q^k-1)/(q-1)$ matrix having as columns a nonzero vector chosen from each one-dimensional subspace of $\F_q^k$.  The \textbf{simplex code} $\mS_q(k)$ of dimension~$k$ over $\F_q$ is a code equivalent to $\rs(G)$.  It is easy to see that all the codewords of $\mS_q(k)$ have the same weight $q^{k-1}$, which implies that $\mS_q(k)$ is a minimal code. Note that this example is the coding theoretic counterpart of Example \ref{ex1}.
\end{example}

Through the same correspondence aforementioned, we get the following.

\begin{theorem}[\cite{alfarano2019geometric,tang2019full}]\label{thm:correspond}
    Equivalence classes of nondegenerate $[n,k]_q$ minimal codes are in one-to-one correspondence with equivalence classes of projective systems which are strong blocking sets in $\PG(k-1,q)$.
\end{theorem}

In \cite{ashikhmin1998minimal} a sufficient condition for an $[n,k]_q$ code to be minimal, known as \textbf{Ashikhmin-Barg (AB) condition}, has been provided. We recall it in the next result.

\begin{lemma}[AB condition {\cite{ashikhmin1998minimal}}]
 Let $\mC$ be an $[n,k,d]_{q}$ code and $w:=\max_{c\in \mC}\wt(\mC)$. If 
\[\frac{w}{d}<\frac{q}{q-1},\] 
then $\mC$ is minimal. 
\end{lemma}

Observing that in a minimal code any nonzero codeword is not only minimal, but also maximal, in \cite{alfarano2022three} the following lower bound on the minimum distance derived.

\begin{theorem}[{\cite[Theorem 2.8]{alfarano2022three}}]\label{thm:mindis}
    Let $\mC$ be an $[n,k,d]_q$ minimal code. Then, 
    $$d\geq (q-1)(k-1)+1.$$
\end{theorem}

Note that the tetrahedron in $\PG(k-1,q)$ gives rise to a minimal code whose minimum distance meets the previous bound with equality. This observation will be useful in Section~\ref{sec:OSBS}.

\section{Concatenation and Field Reduction}\label{sec:fieldreduction}
This section is devoted to the explanation of the \emph{concatenation} method, that is a machinery which can be used to provide families of asymptotically good minimal codes; see \cite{cohen1994intersecting,cohen2013minimal,bartoli2021small}.
The concatenation works as follows. 

Let $\mC$ be an $[N,K,D]_{q^h}$ code and let $(\mI_1,\ldots,\mI_N)$ be a sequence of $[n_\ell,h,d_\ell]_q$ codes, for every $\ell\in[N]$. Let $\boldsymbol{\pi}=(\pi_1,\ldots,\pi_N)$ be a sequence of $\F_q$-linear injective maps, $\pi_\ell:\F_{q^h}\to \F_q^n$, such that $\mI_\ell=\pi_\ell(\F_{q^h})$ for every $\ell\in[N]$. Then, the \textbf{concatenation of} $\mC$ \textbf{with} $(\mI_1,\ldots,\mI_N)$ \textbf{by} $\boldsymbol{\pi}$ is given by $$(\mI_1,\ldots,\mI_N)\square_{\boldsymbol{\pi}} \mC:=\{(\pi_1(c_1),\dots, \pi_N(c_N))\St (c_1,\dots, c_N)\in\mC\}\subseteq \F_q^{M},$$
where $M=\sum\limits_{\ell=1}^N n_\ell$.

The code $\mC$ is called \textbf{outer} code, while the codes $\mathcal{I}_\ell$ are called \textbf{inner} codes. Note that the concatenation depends on the choice of the maps $\pi_\ell$. However, some properties do not depend on it: for example, for every choice of $\pi$ we have that  $(\mI_1,\ldots,\mI_N) \square_{\boldsymbol{\pi}} \mC$ is an $[M,Kh]_q$ code.
If $\boldsymbol{\pi}=(\pi,\ldots,\pi)$ and all the inner codes are the same $[n,h,d]_q$ code~$\mI$, we denote the concatenation $(\mI,\ldots,\mI) \square_{\boldsymbol{\pi}} \mC$ simply by $\mI\square_\pi \mC$. In this case, $\mI\square_\pi \mC$ is an $[Nn, Kh, \geq Dd]_q$ code.
When we consider properties independent of $\boldsymbol{\pi}$, we omit it from the subscript from $\square$.

We finally recall a possible shape of the generator matrix of the concatenated code, corresponding to a specific choice of $\boldsymbol{\pi}$, as described in \cite[Remark 2.4]{bartoli2021small}. Let $G_{\mathcal{I}_1},\ldots, G_{\mathcal{I}_N}\in\F_q^{h\times n}$ be  generator matrices of the inner codes $\mathcal{I}_1,\ldots,\mathcal{I}_N$, respectively, and fix $\omega\in\F_{q^h}$ to be such that $\langle \omega\rangle =\F_{q^h}^\ast$. Let $A\in\F_q^{h\times h}$ be the companion matrix of the minimal polynomial of $\omega$. For any $\alpha \in\F_{q^h}$ let 
$$
A(\alpha) =\begin{cases} 
A^r\in\F_q^{h\times h} & \textnormal{ if } 0\ne\alpha=\omega^r, \\
\mathbf{0}\in\F_q^{h\times h} & \textnormal{ if } \alpha=0.
\end{cases}$$

Consider now a generator matrix $G_{\mC}\in \F_{q^h}^{K\times N}$ of the outer code $\mC$, where $G_{\mC}=(\alpha_{i,j})_{1\leq i\leq K,1\leq j\leq N}$. Then, a generator matrix of a concatenation of $\mathcal{C}$ with $(\mathcal{I}_1,\ldots,\mathcal{I}_N)$ can be chosen as
\begin{equation}\label{eq:gen_matrix_conc}
    \begin{pmatrix}
    A(\alpha_{1,1})G_{\mathcal{I}_1} &  A(\alpha_{1,2})G_{\mathcal{I}_2} & \cdots &  A(\alpha_{1,N})G_{\mathcal{I}_N}\\
    \vdots & \vdots &   & \vdots \\
     A(\alpha_{K,1})G_{\mathcal{I}_1} &  A(\alpha_{K,2})G_{\mathcal{I}_2} & \cdots &  A(\alpha_{K,N})G_{\mathcal{I}_N}
    \end{pmatrix}\in\F_{q}^{Kh \times M}.
\end{equation}

A geometric insight of the concatenation process can be given by making use of the \emph{field reduction map}, which is a map sending subspaces of $\PG(K-1,q^h)$ into subspaces of $\PG(Kh-1,q)$. More formally, let $\mathcal F_q$ be the field reduction map, sending points of $\PG(K-1,q^h)$ in $(h-1)$-dimensional projective subspaces of $\PG(Kh-1,q)$. With a slight abuse of notation, we will consider $\mathcal F_q$ as a multimap from  $\PG(K-1,q^h)$ to $\PG(Kh-1,q)$. 

Let $r>0$ be an integer and $S=\{P_1,\ldots,P_r\}$ be a set of points in $\PG(K-1,q^h)$. We denote by $\mF_q(S)$ the set $\mF_q(P_1)\cup \ldots\cup \mF_q(P_r)$. Let $\mC$ be an $[N,K]_{q^h}$ code and let $\mP(\mC)$ be the projective system associated to one generator matrix for $\mC$. By a little abuse of notation, we denote by $\mF_q(\mC)$ the code associated with the projective system $\mF_q(\mP(\mC))$.

\begin{theorem}\label{thm:field_red+simplex}
 Let $\mC$ be an $[N,K]_{q^h}$ code.
 Then $\mF_q(\mC)=\mS_q(h)\square\mC$.
\end{theorem}
\begin{proof}
    Let $G_\mC=(\alpha_{i,j})_{i,j}$ be a generator matrix for $\mC$ and $A\in\F_{q}^{h\times h}$ be the companion matrix of the minimal polynomial of a primitive element $\omega\in\F_{q^h}^\ast$.
    First of all recall that a generator matrix $G_{\mS_q(h)}$ of $\mS_q(h)$ contains as columns all the points of $\PG(h-1,q)$. Let $G$ be the generator matrix of $\mS_q(h)\square\mC$ given as in Eq. \eqref{eq:gen_matrix_conc}.
    Observe that the first column-block of $G$ is given by $X\cdot G_{\mS_q(h)}$, where
    $$ X=\begin{pmatrix}
    A(\alpha_{1,1}) \\
    A(\alpha_{2,1}) \\
    \vdots \\
    A(\alpha_{K,1})
    \end{pmatrix}.$$
    By definition, $ \mathrm{colsp}(X)$ corresponds to the $(h-1)$-dimensional projective subspace of $\PG(Kh-1,q)$ given by the field reduction of the projective point $P=[\alpha_{1,1}:\ldots:\alpha_{K,1}]$. Finally, the multiplication of $X$ with $G_{\mS_q(h)}$ gives a  nonzero vector of each one-dimensional subspace of $\mathrm{colsp}(X)$. This shows that the first column-block of $G$ is the same as the one of $\mF_q(\mC)$. The same reasoning can be done for all the column-blocks, and hence we get that $\mS_q(h)\square\mC = \mF_q(\mC)$.
\end{proof}

\begin{example}
        Let $\{1,\omega\}$ be a basis of $\F_{2^2}/\F_2$. Let $\mC$ be the code generated by the matrix $$G=\begin{pmatrix}
            1 & 0 & 1 & 1 \\
            0 & 1 & 1 & \omega
        \end{pmatrix}\in\F_{2^2}^{2\times 4}.$$
        Then, $\mF_2(\mC)$ is the code generated by 
        $$ \begin{pmatrix}
        1 & 0 & 1 & 0 & 0 & 0 & 1 & 0 & 1 & 1 & 0 & 1 \\
        0 & 1 & 1 & 0 & 0 & 0 & 0 & 1 & 1 & 0 & 1 & 1 \\
        0 & 0 & 0 & 1 & 0 & 1 & 1 & 0 & 1 & 0 & 1 & 1 \\
        0 & 0 & 0 & 0 & 1 & 1 & 0 & 1 & 1 & 1 & 1 & 0
        \end{pmatrix},
        $$
        which is easily seen to be  a generator matrix for $\mS_2(2)\square\mC$ as given in \eqref{eq:gen_matrix_conc}.
\end{example}

\begin{remark}
    Let $\mC$ be an $[N,K]_{q^h}$ code generated by $G_\mC$ and let $\mI$ be an $[n,h]_{q}$ projective code, different from the simplex code. For each $j\in[N]$, let $\pi_j:\F_{q^h}\to \mI$. Note that since $\mI$ is projective, we have that $\mP(\mI)\subseteq \mP(\mS_q(h))$. For every $j\in[N]$, define $\bar{\pi}_j:\F_{q^h}\to \mS_q(h)$,
    such that when $\bar{\pi}_j$ is applied to the $j$-th entry of $\mC$, it gives the field reduction $\mF_q(P_j)$ of the projective points given by the $j$-th column of $G_\mC$. Applying $\pi_j$ to the $j$-entry of $\mC$ corresponds to selecting only some points from the $(h-1)$-dimensional space given by $\mF_q(P_j)$. In other words, we have that $\mP(\mI\square_{\boldsymbol{\pi}}\mC)\subseteq \mP(\mS_q(h)\square_{\boldsymbol{\bar{\pi}}}\mC)$. We illustrate this in the following example.
\end{remark}

\begin{example}
        Let $\omega\in\F_{2^3}$ be such that $\omega^3=\omega+1$ and let $\{1,\omega,\omega^2\}$ be a basis of $\F_{2^3}/\F_2$. 
        Let $$G_\mC=\begin{pmatrix}
            1 & 0 & 1 & \omega \\
            0 & 1 & 1 & \omega^2
        \end{pmatrix}\in\F_{2^3}^{2\times 4}.$$
        Let $\mI$ be the $[4,3]_2$ code generated by 
        $$ G_\mI=\begin{pmatrix}
            1 & 0 & 0 & 1\\
            0 & 1 & 0 & 1 \\
            0 & 0 & 1 & 1
        \end{pmatrix}$$
        which corresponds to the columns highlighted in yellow of the  generator matrix of the simplex code   
        $$ G_{\mS_2(3)}=\left(\begin{array}{>{\columncolor{yellow!20}}c>{\columncolor{yellow!20}}cc>{\columncolor{yellow!20}}ccc>{\columncolor{yellow!20}}c}
        1 & 0 & 1 & 0 & 1 & 0 & 1 \\
        0 & 1 & 1 & 0 & 0 & 1 & 1 \\
        0 & 0 & 0 & 1 & 1 & 1 & 1
        \end{array}\right).$$
                        Then, we have that 
        $$\tiny{G_{\mS_2(3)\square \mC} =  \left(\begin{array}{>{\columncolor{yellow!20}}c>{\columncolor{yellow!20}}cc>{\columncolor{yellow!20}}ccc>{\columncolor{yellow!20}}c>{\columncolor{yellow!20}}c>{\columncolor{yellow!20}}cc>{\columncolor{yellow!20}}ccc>{\columncolor{yellow!20}}c>{\columncolor{yellow!20}}c>{\columncolor{yellow!20}}cc>{\columncolor{yellow!20}}ccc>{\columncolor{yellow!20}}c>{\columncolor{yellow!20}}c>{\columncolor{yellow!20}}cc>{\columncolor{yellow!20}}ccc>{\columncolor{yellow!20}}c}
            1 & 0 & 1 & 0 & 1 & 0 & 1 & 0 & 0 & 0 & 0 & 0 & 0 & 0 & 1 & 0 & 1 & 0 & 1 & 0 & 1 & 0 & 0 & 0 & 1 & 1 & 1 & 1\\
            0 & 1 & 1 & 0 & 0 & 1 & 1 & 0 & 0 & 0 & 0 & 0 & 0 & 0 & 0 & 1 & 1 & 0 & 0 & 1 & 1 & 1 & 0 & 1 & 1 & 0 & 1 & 0 \\
            0 & 0 & 0 & 1 & 1 & 1 & 1 & 0 & 0 & 0 & 0 & 0 & 0 & 0 &  0 & 0 & 0 & 1 & 1 & 1 & 1 & 0 & 1 & 1 & 0 & 0 & 1 & 1 \\
            0 & 0 & 0 & 0 & 0 & 0 & 0 & 1 & 0 & 1 & 0 & 1 & 0 & 1 & 1 & 0 & 1 & 0 & 1 & 0 & 1 & 0 & 1 & 1 & 0 & 0 & 1 & 1 \\
            0 & 0 & 0 & 0 & 0 & 0 & 0 & 0 & 1 & 1 & 0 & 0 & 1 & 1 & 0 & 1 & 1 & 0 & 0 & 1 & 1 & 0 & 1 & 1 & 1 & 1 & 0 & 0 \\
            0 & 0 & 0 & 0 & 0 & 0 & 0 & 0 & 0 & 0 & 1 & 1 & 1 & 1 & 0 & 0 & 0 & 1 & 1 & 1 & 1 & 1 & 0 & 1 & 1 & 0 & 1 & 0    
        \end{array}\right)}.$$
        A generator matrix for $\mI\square\mC$ is obtained from $G_{\mS_q(h)\square \mC}$ by selecting only the columns highlighted in yellow.
\end{example}

\section{Outer Strong Blocking Sets and Outer Minimal Codes}\label{sec:OSBS}

In this section we introduce the notion of \emph{outer strong blocking sets} and \emph{outer minimal codes}. We establish a series of correspondences and results.
We start by introducing auxiliary definitions.

\begin{definition}
We call a set $\mV\subseteq \F_{q^h}^K$ an $\F_q$\textbf{-vectorial strong blocking set} if   $\langle \mV\cap \mH\rangle_{\Fq}=\mH$ for all  $\F_q$-hyperplanes $\mH$ of $\F_{q^h}^K$.
\end{definition}

\begin{definition}
    Let $\mC$ be an $[N,K]_{q^h}$ code and let $G\in\F_{q^h}^{K\times N}$ be a generator matrix of $\mC$. Let $v_1,\ldots, v_N\in\F_{q^h}^K$ denote the columns of $G$. We define the \textbf{vectorial system associated with $G$} to be the set 
    $$ \mV(G):= \langle v_1\rangle_{\F_{q^h}}\cup  \langle v_2\rangle _{\F_{q^h}}\cup \cdots \cup  \langle v_N\rangle _{\F_{q^h}}.$$

\end{definition}

   Note that if $\mC^\prime$ is a code equivalent to $\mC$ and $G^\prime$ is a generator matrix of $\mC^\prime$, we have that $\mV(G)$ and $\mV(G^\prime)$ are $\GL(K,q^h)$-equivalent. Hence, as we did in Notation \ref{not:PC}, we refer to $\mV(\mC)$ as an arbitrary vectorial system associated to the equivalence class of $\mC$.

\begin{definition}
We say that a subset $\mP$ of $\PG(K-1,q^h)$  is a $q$\textbf{-outer strong blocking set} if $\mF_q(\mP)$ is a strong blocking set in $\PG(Kh-1,q)$, i.e. for any hyperplane $H\subseteq\PG(Kh-1,q)$, we have that
\[\langle \mF_q(\mP)\cap H\rangle =H.\]
From now on, whenever the extension considered is clear from the context, we will simply say outer strong blocking set.
\end{definition}

Recall the following result from \cite{alfarano2022three}.

\begin{lemma}[{\cite[Proposition 4.5]{alfarano2022three}}]\label{lemma_prop4.5}
 Let $U_1\cup \cdots\cup U_r$ be a strong blocking set in $\PG(K-1,q)$. For each $i\in [r]$, let $\Gamma_i:=\langle U_i\rangle \cong\PG(h_i-1,q)$ for some $h_i-1\leq k$ and let $B_i\subseteq\Gamma_i$ be the isomorphic image of a strong blocking set in $\PG(h_i-1,q)$. Then $B_1\cup \cdots\cup B_r$ is a strong blocking set in $\PG(K-1,q)$.
\end{lemma}

\begin{theorem}\label{thm:equivalence}
Let $\mC$ be an $[N,K]_{q^h}$ code. The following are equivalent:
   \begin{itemize}
   \item[{\rm (a)}] $(\mI_1,\ldots,\mI_N)\square \mC$ is minimal, for every sequence $(\mI_1,\ldots,\mI_N)$ of $[n_i,h]_q$ minimal codes.
   \item[{\rm (b)}] $\mS_q(h)\square \mC$ is minimal.
   \item[{\rm (c)}] $\mP(\mC)$ is an outer strong blocking set.
   \item[{\rm (d)}] $\mV(\mC)$ is an $\F_q$-vectorial strong blocking set.
   \end{itemize}
\end{theorem}

\begin{proof}
\underline{{\rm (a)}$\Rightarrow${\rm (b)}.} This is trivially true, since $\mS_q(h)$ is minimal.

\noindent \underline{{\rm (b)}$\Rightarrow${\rm (a)}.} Since  $\mS_q(h)\square \mC$ is minimal, the set $\mP(\mS_q(h)\square \mC)$ is a strong blocking set. Each $\mP(\mI_j)$ is a strong blocking set in $\mP(\mS_q(h))\cong\PG(h-1,q)$. Hence, $\mP((\mI_1,\ldots,\mI_N)\square \mC)$ is a strong blocking set by Lemma \ref{lemma_prop4.5}. 

\noindent \underline{{\rm (b)}$\Leftrightarrow${\rm (c)}.} This is a direct consequence of Theorem \ref{thm:correspond} and Theorem \ref{thm:field_red+simplex}.

\noindent \underline{{\rm (c)}$\Leftrightarrow${\rm (d)}.} 
This equivalence follows immediately from the correspondence between $\F_q$-hyperplanes in $\F_{q^h}^K$ and hyperplanes in $\PG(Kh-1,q)$, and from the isomorphism between $\F_{q^h}^K$ and $\F_q^{Kh}$ as $\Fq$-vector spaces.
\end{proof}

\begin{definition}
We call $q$\textbf{-outer minimal} an $[N,K]_{q^h}$ code $\mC$ satisfying one of the equivalent conditions in Theorem \ref{thm:equivalence}. Again, we will omit the $q$- whenever the extension considered is clear from the context.
\end{definition}

\begin{remark}
Note that if there exists $j$ such that $\mI_j$ is not minimal, then  $(\mI_1,\ldots,\mI_N)\square \mC$ is not necessarily minimal.
For instance, let $\F_4=\F_2(\alpha)$ with $\alpha^2=\alpha+1$ and consider $\mC=\mS_4(2)\subseteq \F_4^5$. Let $\mI$ be the whole space $\F_2^2$, which is not minimal. Then a generator matrix for $\mI\square \mC$ is given by
$$\begin{pmatrix}
1 & 0 & 0 & 0 & 1 & 0 & 1 & 0 & 1 & 0 \\
0 & 1 & 0 & 0 & 0 & 1 & 0 & 1 & 0 & 1 \\
0 & 0 & 1 & 0 & 1 & 0 & 0 & 1 & 1 & 1 \\
0 & 0 & 0 & 1 & 0 & 1 & 1 & 1 & 1 & 0
\end{pmatrix}\in\F_2^{4\times {10}}.$$
Clearly, $\mI\square \mC$ is not minimal, since the first two rows of its generator matrix have nonintersecting support.
\end{remark}

In the following we characterize outer minimal codes in terms of the supports of their codewords, in the same way as it is defined for minimal codes. We first introduce  the concept of \emph{outer minimal codeword}.

\begin{definition}\label{def:outercodeword}
   Let $\mC$ be an $[N,K]_{q^h}$ code. A nonzero codeword $c\in \mC\subseteq  \F_{q^h}^N$ is called ($q$-)\textbf{outer minimal} if, for all $c'\in \mC$, \
     \[\sH(c')\subseteq \sH(c) \wedge \forall i\in \sH(c), \exists \lambda_i\in \F_q \text{ s.t. } c'_i=\lambda_i c_i \ \ \ \Longrightarrow \ \ \ \exists \lambda\in \F_q \text{ s.t. } c'=\lambda c.\] 
\end{definition}

\begin{proposition}\label{prop:outercodeword}
    An $[N,K]_{q^h}$ code $\mC$ is outer minimal if and only if all its nonzero codewords are outer minimal.
\end{proposition}

\begin{proof}
    ($\Rightarrow$) Assume $\mC$ is outer minimal. Let $\pi:\F_{q^h}\to \F_q^{\frac{q^h-1}{q-1}}$ be an $\F_q$-linear map such that $\pi(\F_{q^h})=\mS_q(h)$. Consider the concatenated code \begin{equation}\label{eq:conc_simplex}
        \mS_q(h)\square_\pi \mC=\{(\pi(c_1),\ldots,\pi(c_N))\St c=(c_1,\ldots,c_N)\in\mC\}.
    \end{equation}
    Then, $\mS_q(h)\square_\pi \mC$ is a minimal code by Theorem \ref{thm:equivalence}. Hence, for every $v,v^\prime\in\mS_q(h)\square_\pi \mC$ such that $\sH(v^\prime)\subseteq \sH(v)$, there is a scalar $\lambda\in\F_q$ such that $v^\prime=\lambda v$. In particular, there are two codewords $c,c^\prime\in\mC$ such that 
    $v=(\pi(c_1),\ldots,\pi(c_N))$ and $v^\prime=(\pi(c_1'),\ldots,\pi(c_N'))$ and $\pi(c_i')=\lambda\pi(c_i)$ for every $i\in\{1,\ldots,N\}$. Since $\pi$ is an $\F_q$-linear map and it sends nonzero elements into nonzero elements, we have that $\sH(c')\subseteq\sH(c)$ and $c'=\lambda c$. In other words, the codewords of $\mC$ are all outer minimal according to Definition \ref{def:outercodeword}.
    
    \noindent
     ($\Leftarrow$) Assume that all the codewords of $\mC$ are outer minimal and consider the concatenation $\mS_q(h)\square_\pi \mC$ as in   \eqref{eq:conc_simplex}. Consider $v=\pi(c)$, $v'=\pi(c') \in\mS_q(h)\square_\pi \mC$, for $c,c^\prime\in\mC$ such that $\sH(v')\subseteq\sH(v)$. Then, by the linearity of $\pi$, we have that $\sH(c')\subseteq\sH(c)$. By assumption, if for every $i \in\sH(c)$ there is $\lambda_i\in\F_q$ such that $c_i^\prime=\lambda_i c_i$, then there is $\lambda\in\F_q$ such that $c'=\lambda c$. This means that $v'=\pi(c')=\lambda\pi(c)=\lambda v$, i.e. $\mS_q(h)\square_\pi \mC$ is minimal and by Theorem \ref{thm:equivalence}, $\mC$ is outer minimal.
\end{proof}

As an easy consequence, we get the following result, which is a slight generalization of \cite[Theorem 2.2]{bartoli2021small}. 

\begin{corollary}[Outer AB condition]
Let $\mC$ be an $[N,K,D]_{q^h}$ code and $W:=\max_{c\in \mC}\wt(\mC)$. If 
\[\frac{W}{D}<\frac{q}{q-1},\] 
then $\mC$ is outer minimal. 
\end{corollary}

\begin{remark}
    Let $\mC$ be an $[N,K,D]_{q^h}$ code satisfying the AB condition. Then $\mC$ satisfies the outer AB condition.
\end{remark}

\begin{proposition}
Let $\mC$ be an $[N,K]_{q^h}$ minimal code. Then, $\mC$ is outer minimal. Equivalently, if $\mP\subseteq \PG(K-1,q^h)$ is a strong blocking set, then $\mP$ is an outer strong blocking set. 
\end{proposition}
\begin{proof}
Since $\mC$ is an $[N,K]_{q^h}$ minimal code, every nonzero codeword $c$ in $\mC$ is minimal. In particular, for all $c'\in\mC$ with $\sigma(c')\subseteq \sigma(c)$, there exists $\lambda\in\F_{q^h}$ such that $c'=\lambda c$. Assume that in addition for every $i\in\sigma(c)$, there is $\mu_i\in\F_q$ such that $c_i'=\mu_i c_i$. Then, we have that $\mu_i=\lambda$ for every $i\in\sigma(c)$, and hence $\lambda\in\F_q$. This implies that $c$ is outer minimal. By Proposition \ref{prop:outercodeword}, this shows that $\mC$ is outer minimal.
\end{proof}

In Figure \ref{fig:graphic} an illustrative summary of the above results is provided.

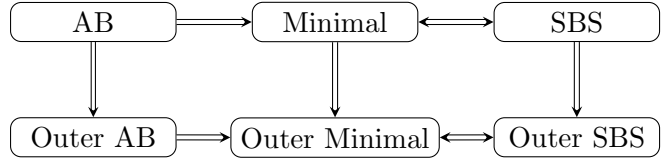
\begin{figure}[ht!]
    \centering
    \begin{tikzpicture}[
node distance = 10mm and 10mm, 
  box/.style = {draw, rounded corners, 
                 minimum width=22mm, minimum height=4mm, align=center},
                  > = {stealth},
  bend angle = 15,
         auto = right,
                        ]
\node (n1)  [box] {AB};
\node (n2)  [box, below=of n1]    {Outer AB};
\node (n3)  [box, right=of n1]    {Minimal};
\node (n4)  [box, below=of n3]    {Outer Minimal};
\node (n5)  [box, right=of n3]   {SBS}; 
\node (n6)  [box, below=of n5]   {Outer SBS};

\draw[double distance=1.2pt,-{Implies},double,->] (n1) to  (n2);
\draw[double distance=1.5pt,double,->] (n1) to  (n3);
\draw[double distance=1.5pt,double,->] (n3) to (n4);
\draw[double distance=1.5pt,double,->] (n2) to (n4);
\draw[double distance=1.5pt,double,<->] (n3) to (n5);
\draw[double distance=1.5pt,double,<->] (n4) to (n6);
\draw[double distance=1.5pt,double,->] (n5) to (n6);
     \end{tikzpicture}
 
    \caption{Relations.}
    \label{fig:graphic}
\end{figure}

We conclude the section with a lower bound on the minimum distance of outer minimal codes. 

\begin{theorem}\label{thm:minD}
Let $\mathcal{C}$ be an outer minimal $[N,K]_{q^h}$ code. Then    
\[d(\mathcal{C})\geq \left\lceil \frac{(q-1)(Kh-1)+1}{(q-1)(h-1)+1}\right\rceil.\]
\end{theorem}

\begin{proof}
Let $\mI$ be an $h$-dimensional code over $\F_q$. It is clear that there exists an appropriate choice of $\pi_1,\ldots,\pi_N$ such that the concatenation $\mI\square_{\boldsymbol{\pi}}\mC$ has exactly minimum distance $d(\mathcal{C})\cdot d(\mathcal{I})$. This is obtained by taking a nonzero codeword $c$ in $\mathcal{C}$ of minimum weight and choosing $\pi_1,\ldots,\pi_N$
so that the image $\pi_i(c_i)$ of each nonzero entry $c_i$ of $c$ has minimum weight in $\mathcal I$. Now, suppose that $\mathcal{C}$ is outer minimal, so that $\mathcal{I} \square \mathcal{C}$ is minimal. Then, by Theorem \ref{thm:mindis},
\[d(\mathcal{I} \square_{\boldsymbol{\pi}} \mathcal{C})\geq (q-1)(Kh-1)+1.\]
So 
\[d(\mathcal{C})\geq \left\lceil\frac{(q-1)(Kh-1)+1}{d(\mathcal{I})}\right\rceil\]
If we choose $\mathcal{I}$ to be the tetrahedron, we get the statement.
\end{proof}

\begin{remark}
 Consider the case $K=2$ in Theorem \ref{thm:minD}. We get that the minimum distance of an outer minimal $[N,2]_{q^h}$ code $\mC$ is at least 
 $$ \left\lceil \frac{(q-1)(2h-1)+1}{(q-1)(h-1)+1}\right\rceil=2.$$
 If we assume that the code $\mC$ is projective, then the associated projective $[N,2]_{q^h}$ system is an arc in the projective line $\PG(1,q^h)$. Thus, the code $\mC$ is an MDS code, and its minimum distance is $N-1$. Since adding points with multiplicity to the projective system can only increase the minimum distance of the associated code, any outer minimal $[N,2]_{q^h}$ code has distance at least 
 $$\min\{ R \in \mathbb N \St \mbox{ there exists an outer strong blocking set in $\PG(1,q^h)$ of size } R \} - 1.$$
  We will see in the next section (Remark \ref{rem:minlength}) that this is equal to
  $$ \begin{cases} 3 & \mbox{ if }q>2, \\
  2 & \mbox{ if } q=2.\end{cases}$$
  Thus, for $K=2$ the lower bound of Theorem \ref{thm:minD} is attained only for $q=2$.
  
  If we consider the case $K=3$ and $h=2$, from Theorem \ref{thm:minD} we get that the minimum distance of an outer minimal $[N,3]_{q^2}$ code $\mC$ is at least 
  $$\left\lceil\frac{5(q-1)+1}{(q-1)+1}\right\rceil=\begin{cases} 5 & \mbox{ if } q>4, \\
  4 & \mbox{ if } q\in\{3,4\},\\
 3 & \mbox{ if } q=2. 
  \end{cases}$$
  For $q>2$, the construction in \cite[Theorem 3.15]{bartoli2022cutting}, provides an outer minimal $[7,3,5]_{q^2}$ code; see also Theorem \ref{thm:cutting_sevenlines}. Thus, when $K=3$ and  $h=2$, the bound of Theorem \ref{thm:minD} is attained  for every $q>4$.
\end{remark}

\section{Bounds of the Size of Outer Strong Blocking Sets}\label{sec:bounds}

Adding a point to an outer strong blocking set gives a larger outer strong blocking set. For this reason, it is natural to look for small sets with this property. In particular, in this section we will present lower and upper bounds on the cardinality of the smallest outer strong blocking sets for a given dimension and ground field. 

\subsection{Lower Bound}

We start by introducing a general property on a set of subsets of $\PG(k-1,q)$ such that their union is a strong blocking set.  This is a generalization of the property on sets of lines given by Fancsali and Sziklai in \cite[Theorem 11]{fancsali2014lines}.
\begin{definition}
    Let $U_1,\ldots, U_N$  be subsets of $\PG(k-1,q)$ and let $(h_i-1)=\dim(\langle U_i\rangle)$ for $i\in [N]$. We say that the collection $\{U_1,\ldots, U_N\}$ has the \textbf{avoidance property} if there exists no codimension $2$ subspace $\Lambda$ of $\PG(k-1,q)$ such that 
    $$ \dim(\langle \Lambda\cap U_i\rangle )\geq h_i-2, \quad \textnormal{ for every } i \in [N].$$
\end{definition}

\begin{theorem}
    If $\{U_1,\ldots, U_N\}$ is a collection of subsets of $\PG(k-1,q)$ having the avoidance property, then $U:=\bigcup\limits_{i=1}^N U_i$ is a strong blocking set.
\end{theorem}

\begin{proof}
Let $h_i-1=\dim(\langle U_i\rangle)$ for $i\in [N]$. Let $H$ be a hyperplane in $\PG(k-1,q)$. Assume by contradiction that $H \cap U$ is contained in a $(k-3)$-dimensional subspace $\Lambda$ of $\PG(k-1,q)$. Then, $\Lambda$ contains each $ \langle H \cap U_i\rangle$. However, $\dim(\langle H\cap U_i\rangle)\geq \dim(\langle U_i\rangle)-1=h_i-2$ for each $i\in[N]$, and this contradicts the hypothesis of $\{U_1,\ldots,U_N\}$ having the avoidance property.
\end{proof}

Now, we will restrict our study to the case of collections of subspaces with the same dimension. 
Moreover, we will distinguish between collections of subspaces of $\PG(k-1,q)$ having the avoidance property and those not having this property, but whose union is anyway a strong blocking set. In both cases, we will give a lower bound on their size, by generalizing the arguments given in \cite[Theorem 14]{fancsali2014lines}. In particular, the following two results are a generalization of \cite[Lemma 12]{fancsali2014lines} and \cite[Lemma 13]{fancsali2014lines}, respectively.

\begin{lemma}\label{lem:number_hspaces_q+1}
 Let $\{U_1, \ldots,  U_N\}$ be a collection of $(h-1)$-dimensional subspaces of $\PG(k-1,q)$ not satisfying the avoidance property. If $U:=\bigcup\limits_{i=1}^N U_i$ is a  strong blocking set, then $N\geq q+1$. 
\end{lemma}
\begin{proof}
    Since $\{U_1, \ldots,  U_N\}$  does not have the avoidance property, there is a codimension~$2$ subspace $\Lambda$ of $\PG(k-1,q)$ such that 
    $$ \dim(\Lambda\cap U_i)\geq h-2, \quad \textnormal{ for every } i \in [N].$$
    Let $\ell=\{P_1,\ldots,P_{q+1}\}$ be a line disjoint from $\Lambda$. For any $i\in[q+1]$, consider the hyperplane $\Pi_i$ spanned by $\Lambda$ and $P_i$. Since $U$ is a strong blocking set, there exists a point~$Q_i$ such that
    $$ Q_i\in(\Pi_i\cap U)\setminus \Lambda.$$
    Hence, there exist $s_i\in[N]$, such that $Q_i\in U_{s_i}\setminus \Lambda$.  Therefore, $U_{s_i}\subseteq \Pi_i$, but $U_{s_i}\not\subseteq \Lambda$. Now, for $j \in[q+1]$ with $j\neq i$, observe that $U_{s_j}\subseteq \Pi_j$ and $U_{s_j} \not\subseteq \Lambda=\Pi_i\cap\Pi_j$. Thus, necessarily $U_{s_i}\neq U_{s_j}$, and $N$ has to be at least the number of points in $\ell$, i.e. $q+1$.
\end{proof}

\begin{theorem}\label{thm:number_hspaces}
    The number of $(h-1)$-dimensional subspaces in $\PG(k-1,q)$ whose union is a strong blocking set is at least
    \[\min\left\{\left\lfloor \frac{k-1}{h}\right\rfloor+\left\lfloor\frac{k-2}{h-1}\right\rfloor+1,q+1\right\}.\]
\end{theorem}

\begin{proof}
    Consider the set $\mathcal U=\left\{U_1,\ldots,U_{\left\lfloor \frac{k-1}{h}\right\rfloor},U_{\left\lfloor \frac{k-1}{h}\right\rfloor+1},\ldots, U_{\left\lfloor \frac{k-1}{h}\right\rfloor+\left\lfloor \frac{k-2}{h-1}\right\rfloor}\right\}$ consisting of $(h-1)$-dimensional subspaces of $\PG(k-1,q)$ whose union is a strong blocking set. By an easy dimension argument, we can deduce that there exists a hyperplane $\Pi$ of $\PG(k-1,q)$ containing $U_1,\ldots,U_{\left\lfloor \frac{k-1}{h}\right\rfloor}$. For $i\in\left[\left\lfloor \frac{k-2}{h-1}\right\rfloor\right]$, let 
    \[W_i\subseteq \Pi \cap U_{\left\lfloor \frac{k-1}{h}\right\rfloor+i}\]
    be an $(h-2)$-dimensional subspace. Again by an easy dimension argument, there exists $H\subseteq \Pi$ of dimension $k-2$ containing each $W_i$. Since $H$ is a hyperplane in $\Pi$, it intersects each $U_j$, $j\in\left[\left\lfloor \frac{k-1}{h}\right\rfloor\right]$, in a space of dimension at least $h-2$. Moreover, it intersects all the others in a space of dimension at least $h-2$ by construction. So $\mathcal U$ does not satisfy the avoidance property. Hence, its cardinality must be greater than $q+1$ by Lemma \ref{lem:number_hspaces_q+1}.
\end{proof}

The following is a generalization of \cite[Theorem 3.12]{heger2021short}. However, in contrast to the latter,  we do not make use of the Jamison's bound \cite{jamison1977covering} on the cardinality of affine blocking sets, since it gives a weaker result. 

\begin{theorem}
    The number of $(h-1)$-dimensional subspaces in $\PG(k-1,q)$ whose union is a strong blocking set is at least
    \[\left\lfloor \frac{k-1}{h}\right\rfloor+\left\lceil \frac{k}{h}\right\rceil.\]
\end{theorem}

\begin{proof}
Let $\mathcal{S}$ be a set of $(h-1)$-dimensional subspaces of $\PG(k-1,q)$ whose union is a strong blocking set and suppose that a hyperplane $H\subseteq \PG(k-1,q)$ contains $t$ many $(h-1)$-dimensional subspaces of $\mathcal{S}$. Let us call $\mathcal T$ the set of such $t$ many $(h-1)$-dimensional subspaces. Let $S:=\bigcup_{s\in \mathcal{S}} s$. Then $S\setminus H$
is an affine  blocking set in $\PG(k-1,q)\setminus H\cong \mathrm{AG}(k-1,q)$. Moreover, since every $(h-1)$-dimensional subspace in $\mathcal T$ is entirely contained in $H$ we have
$$S\setminus H=\Big(\bigcup_{s\in \mathcal S\setminus\mathcal T}s\Big)\setminus H=\bigcup_{s\in\mathcal S\setminus \mathcal T}(s\setminus H). $$
Thus, $S\setminus H$ is an affine blocking set consisting of  the union of $|\mathcal S\setminus\mathcal T|=|\mathcal S|-t$ many $(h-1)$-dimensional affine subspaces. Since an affine blocking set cannot be contained in an affine hyperplane of $\mathrm{AG}(k-1,q)$, we need at least $\lceil \frac{k}{h}\rceil$ many $(h-1)$-dimensional affine subspaces. This means $|\mathcal S|-t\geq \lceil \frac{k}{h}\rceil$.
As $t$ many $(h-1)$-dimensional subspaces span a subspace of dimension at most $th-1$, we can find a hyperplane $H$ which contains at least $\left\lfloor \frac{k-1}{h}\right\rfloor$ many $(h-1)$-dimensional subspaces of $\mathcal{S}$ (note that we need $|\mathcal{S}|\geq \left\lfloor \frac{k-1}{h}\right\rfloor$ but this is clear because $S$ cannot be contained in a hyperplane). Therefore, putting everything together, we obtain
\[|\mathcal{S}|=t+|\mathcal S\setminus \mathcal T|\geq \left\lfloor \frac{k-1}{h}\right\rfloor+\left\lceil \frac{k}{h}\right\rceil. \qedhere\]
\end{proof}

\begin{remark}
 Let us underline that the results obtained so far in this section, are valid for ge\-neral sets of (not necessarily disjoint) subspaces of $\PG(k-1,q)$ and not only those belonging to a Desarguesian spread, that is, those coming from the field reduction map from $\PG\left(\frac{k}{h}-1,q^h\right)$. 
\end{remark}

The correspondence between outer strong blocking sets and outer minimal codes implies the following corollary.

\begin{corollary}\label{cor:bound_length}
The cardinality $N$ of an outer strong blocking set in $\PG(K-1,q^h)$, or equivalently the length $N$ of a $K$-dimensional outer minimal code defined over $\F_{q^h}$, satisfies the following inequality
    \[N\geq \max\left\{\min\left\{2K+\left\lfloor \frac{K-2}{h-1}\right\rfloor,q+1\right\},2K-1\right\}.\]
\end{corollary}

\begin{remark}\label{rem:minlength}
 The bound in Corollary \ref{cor:bound_length} for $K=2$ is tight for every $q$. Indeed, in this case we have that an outer strong blocking set in $\PG(1,q^h)$ (equivalently, an outer minimal $[N,2]_{q^h}$ code) has size (respectively, length) at least 
 $$N \geq \begin{cases} 4  & \mbox{ if } q>2, \\
 3  & \mbox{ if } q=2.
 \end{cases}
 $$
 For $q>2$, the construction given in \cite[Theorem 4.2]{alfarano2022three} (see Theorem \ref{thm:4points} for a concrete explanation) provides an outer strong blocking set consisting of $4$ points. Furthermore, when $q=2$, it can be shown that any $3$ points in $\PG(1,2^h)$ form an outer strong blocking set. Indeed, every set of three points in $\PG(1,2^h)$ is $\mathrm{PGL}(2,2^h)$-equivalent to $\{[1:0], [0:1], [1:1]\}$, and it can be verified straightforwardly that their field reduction is
 $$ \{[a:0] \St a \in \PG(h-1,2)\}\cup\{[0:a] \St a \in \PG(h-1,2)\} \cup\{[a:a] \St a \in \PG(h-1,2)\},$$and it is a strong blocking set in $\PG(2h-1,2)$.
\end{remark}

\subsection{Existence Results}

We recall that, by Proposition \ref{prop:outercodeword}, a code is outer minimal if and only if all its nonzero codewords are outer minimal. We will use this property to get an upper bound on the length of the shortest outer minimal code for given dimension and ground field. As a byproduct, we will get the best-known general upper bound on the cardinality of the smallest strong blocking set. We would like to emphasize that our proof is genuinely coding-theoretical and it improves on previous results proved with geometric arguments in~\cite{heger2021short}. The proof of the following theorem is inspired by \cite{chabanne2013towards}.  

\begin{theorem}
    If 
    \[\binom{N-2}{K-2}_{q^h}\cdot \sum_{i=1}^N \binom{N}{i}(q^h-1)^i\left(q^i-q\right)<\binom{N}{K}_{q^h},\]
    then there exists an $[N,K]_{q^h}$ outer minimal code. In particular, $[N,K]_{q^h}$ outer minimal codes exist whenever
    \[N\geq \left\lceil \frac{2}{\log_{q^h}\left(\frac{q^{2h}}{q^{h+1}-q+1}\right)}\cdot K\right\rceil \]
\end{theorem}
\begin{proof}
 A nonzero codeword $c\in \mC\subseteq  \F_{q^h}^N$ is not outer minimal if there exists $c'\in \mC$ such that
    \begin{equation}\label{eq:notouter} c'\neq \lambda c, \ \forall \lambda\in \F_q \ \ \wedge \ \  \sH(c')\subseteq \sH(c) \ \ \wedge \ \ \forall i\in \sH(c), \ \exists \lambda_i\in \F_q \text{ s.t. } c'_i=\lambda_i c_i.\end{equation} 
Let us define 
\[B=\left\{(c,c')\in \F_{q^h}^N \times \F_{q^h}^N \St c\neq 0 \text{ and }\eqref{eq:notouter} \text{ holds}\right\}.\]
The cardinality of $B$ is given by
    $$|B|=\sum_{i=1}^N \binom{N}{i}(q^h-1)^i\left(q^i-q\right).$$
Since the $2$-dimensional subspace generated by each pair in $B$ is contained in exactly $\binom{N-2}{K-2}_{q^h}$ codes of dimension $K$ in $\F_{q^h}^N$ and the total number of $[N,K]_{q^h}$ codes is $\binom{N}{K}_{q^h}$, we have that if $\binom{N-2}{K-2}_{q^h}\cdot |B| < \binom{N}{K}_{q^h}$ then there exists an $[N,K]_{q^h}$ code $\mC$ which does not contain any $2$-dimensional subspace generated by an element of $B$. Therefore, $\mC$ is outer minimal.\\

To obtain the second inequality, it is enough to observe that 
\[|B|< (1+(q^h-1)q)^N\]
and
\[\frac{\binom{N}{K}_{q^h}}{\binom{N-2}{K-2}_{q^h}}=\frac{(q^{Nh}-1)(q^{Nh}-q^h)}{(q^{Kh}-1)(q^{Kh}-q^h)}>q^{2h(N-K)}.\]
\end{proof}

Due to the connection between outer minimal codes and outer strong blocking sets, we get immediately the following.

\begin{theorem}
There exist 
\[\left\lceil \frac{2}{\log_{q^h}\left(\frac{q^{2h}}{q^{h+1}-q+1}\right)}\cdot K\right\rceil\]
many $(h-1)$-dimensional subspaces of $\PG(Kh-1,q)$ whose union is a strong blocking set. In particular
\[m(Kh,q)\leq \left\lceil \frac{2}{\log_{q^h}\left(\frac{q^{2h}}{q^{h+1}-q+1}\right)}\cdot K\right\rceil \cdot m(h,q)\]
and for $k$ even
\[m(k,q)\leq \left\lceil \frac{2k}{\log_{q}\left(\frac{q^4}{q^3-q+1}\right)}\right\rceil\cdot (q+1).\]
\end{theorem}

\begin{remark}
Considering quadratic extensions of $\F_q$, we get that there exist
\[\left\lceil \frac{1}{\log_{q^2}\left(\frac{q^4}{q^3-q+1}\right)}\cdot 2K\right\rceil\]
many lines in $\PG(2K-1,q)$ whose union is a strong blocking set, or, equivalently, the length of the shortest minimal code of even dimension  $k$ over $\F_q$ satisfies
\[m(k,q)\leq \left\lceil \frac{1}{\log_{q^2}\left(\frac{q^4}{q^3-q+1}\right)}\cdot k\right\rceil\cdot (q+1).\]
For odd dimensions, it is enough to shorten the code on a coordinate to get essentially the same bound. Note that
\[\frac{1}{\log_{q^2}\left(\frac{q^4}{q^3-q+1}\right)}<\frac{2}{1+\frac{1}{(q+1)^2\ln q}}.\]
Indeed, the above inequality is equivalent to 
\[\ln \left(1+\frac{q-1}{q^3-q+1}\right)>\frac{1}{(q+1)^2},\]
which is true for any $q\geq 2$ using the fact that $\ln (1+x)>x/(x+1)$ for $x>-1$ and $x\neq 0$.\\
Hence, for $k$ large enough and $q\neq 2$, our bound is better than those proved in \cite{chabanne2013towards,heger2021short} and it is then, to the best of our knowledge, the best-known upper bound on the cardinality of the smallest strong blocking set in $\PG(k-1,q)$.
\end{remark}

\begin{remark}
Independent of our work, Bishnoi, D'haeseleer, Gijswijt and Potukuchi have obtained the same upper bound in \cite{bishnoi2023blocking}. Moreover, one of the authors of \cite{heger2021short} pointed out to us that this bound could be derived from their paper as well, using different methods.
\end{remark}

\subsection{Small (\emph{k-2})-saturating sets}

As we mentioned already, strong blocking sets have been introduced in  \cite{davydov2011linear} to construct small saturating sets. In this short subsection, we highlight the consequences of the bounds proved above on saturating sets.

\begin{definition}
A set of points $S \subseteq  \mathrm{PG}(k-1, q)$ is $\rho$-\textbf{saturating} if any point $Q \in \mathrm{PG}(k-1, q)$ belongs to the space generated by $\rho+1$ points of $S$ and $\rho$ is the smallest value with this property. We denote $s_q(k-1,\rho)$ the smallest size of a $\rho$-saturating set in $\mathrm{PG}(k-1, q)$.
\end{definition}

Clearly every strong blocking set is also $(k-2)$-saturating but this gives in general large saturating sets. A refined connection between strong blocking sets and $(k-2)$-saturating sets is the following. 

\begin{theorem}[{\cite[Theorem 3.2]{davydov2011linear}}]
Any strong blocking set in a subgeometry $\mathrm{PG}(k-1,q)$ of  $\mathrm{PG}(k-1,q^{k-1})$ is a $(k-2)$-saturating set in $\mathrm{PG}(k-1,q^{k-1})$.
\end{theorem}

Recent improvements on the upper bound on the minimum size of a $\rho$-saturating set have been obtained in \cite{Davydov_2019,denaux2021constructing}. However, for 
$\rho=k-2$, the best-known results are given by \cite[Corollaries 6.4 and 6.7]{heger2021short}. The following is an improvement on the previous results for $k$ large and $q\neq 2$.

\begin{corollary} Let $j=0$ if $k$ is even and $j=1$ otherwise. Then
    \[s_{q^{k-1}}(k-1,k-2)\leq \left\lceil \frac{1}{\log_{q^2}\left(\frac{q^4}{q^3-q+1}\right)}\cdot (k+j)\right\rceil\cdot (q+1).\]
\end{corollary}

\section{Characterization and Construction  of Outer Strong Blocking Sets}\label{sec:constructions}

This section is dedicated to constructing outer minimal codes by iterative concatenations. We do it geometrically, constructing outer strong blocking sets. At each step we start with a strong blocking set over an extension field, and we select a special subset that will result in an outer strong blocking set. The main idea is that at each step, we select a strong blocking set given as union of projective subspaces. Note that we do not require it to satisfy the avoidance property. In order to do so, we first recall two results, one from~\cite{alfarano2022three} and one from \cite{bartoli2022cutting}, and then generalize the latter.

\medskip

The first idea is the following. Given a union of subsets which forms a strong blocking set, one can consider for each subset its projective span, and take inside it a strong blocking set. The union will be a strong blocking set, as shown in  Lemma \ref{lemma_prop4.5}. 

\medskip

The second idea is given in the following result, which provides a general construction of outer strong blocking sets starting from a strong blocking set as union of lines. This can be straightforwardly deduced from  \cite{alfarano2022three}; see also \cite[Theorem 3.7]{davydov2011linear} for the special case $h=2$ and $K=2$.

\begin{theorem}[{\cite[Theorem 4.2]{alfarano2022three}}]\label{thm:4points}
 Let $\mL=\{\ell_1,\ldots,\ell_r\}$ be a set of lines  of~$\PG(K-1,q^h)$ whose union is a strong blocking set. For each $i\in [r]$, let $\alpha_{1}^{(i)},\alpha_{2}^{(i)},\alpha_{3}^{(i)},\alpha_{4}^{(i)}$ be four points of $\ell_i$ not lying on a proper subline. Then, 
 $$\bigcup_{i=1}^r \{\alpha_{1}^{(i)},\alpha_{2}^{(i)},\alpha_{3}^{(i)},\alpha_{4}^{(i)}\}$$
 is an outer strong blocking set.
\end{theorem}

When $h=2$, an example of outer blocking set of $\PG(2,q^2)$ was found in \cite{bartoli2022cutting}. Here we give a more general statement, that combines the result in \cite[Theorem 3.15]{bartoli2022cutting} with Lemma~\ref{lemma_prop4.5}.

\begin{theorem}[{\cite[Theorem 3.15]{bartoli2022cutting}}]\label{thm:cutting_sevenlines}
 Let $\{\Pi_1,\ldots,\Pi_r\}$ be a set of planes in $\PG(K-1,q^2)$ whose union is a strong blocking set. Then there exist seven points $\{\beta_1^{(i)},\ldots,\beta_7^{(i)}\}$ on each plane $\Pi_i$ such that 
 $$\bigcup_{i=1}^r \{\beta_1^{(i)},\ldots,\beta_7^{(i)}\}$$
 is an outer strong blocking set.
\end{theorem}

\subsection{The Avoidance Property: Linear Sets and Hermitian Varieties} 
In this subsection, we generalize the results of Theorem \ref{thm:cutting_sevenlines} (\cite[Theorem 3.15]{bartoli2022cutting}) for general sets of subspaces and over general degree $h$ extensions. We are not able to provide an existence result, but we give an equivalent characterization on a set of points satisfying the avoidance property, involving  $q$-systems, linear sets and (degenerate) Hermitian varieties.

\begin{definition}
An $[n,K]_{q^h/q}$\textbf{system} $V$ is an $n$-dimensional $\Fq$-subspace of $\F_{q^h}^K$ such that $\langle V \rangle_{\F_{q^h}}=\F_{q^h}^K$. When we do not care about the parameters, we simply refer to $V$ as a $q$-\textbf{system}.
\end{definition}

From a geometric point of view, $q$-systems define linear sets. These are well-studied objects in finite geometry, which generalize the concept of subgeometry. Their name was introduced by Lunardon in \cite{lunardon1999normal}, where linear sets are used for special constructions of
blocking sets. The interested reader is referred to \cite{polverino2010linear} for a survey on linear sets.

\begin{definition}
Let $V$ be an $[n,K]_{q^h/q}$ system. The \textbf{linear set} associated to $V$ is a pair $(L_{V},\wt_{V})$, with
$$L_{V}=\{\langle v\rangle_{\F_{q^h}} \St v \in V\}\subseteq \PG(K-1,q^h),$$
and $\wt_{V}$ -- called the \textbf{weight function} -- defined as
$$ \begin{array}{rccl }\wt_{V}:& \PG(K-1,q^h) & \longrightarrow & \mathbb N \\
& P=\langle u\rangle_{\F_{q^h}} & \longmapsto & \dim_{\Fq}(V\cap \langle u\rangle_{\F_{q^h}}).
\end{array}$$
The \textbf{rank} of $(L_{V},\wt_{V})$ is $\dim_{\Fq}(V)=n$.
\end{definition}

With these notions in mind, we can give a new characterization of outer strong blocking sets, in terms of linear sets. The following is a generalization of \cite[Lemma 3.8]{bartoli2022cutting}. This will be clearer by reading the result afterwards.

\begin{theorem}\label{thm:characterization_outer_strong} 
 Let $\mP=\{P_1,\ldots,P_N\}\subseteq \PG(K-1,q^h)$ be a set of points. The following are equivalent.
 \begin{enumerate}
      \item[{\rm (a)}] $\{\mF_q(P_1),\ldots,\mF_q(P_N)\}$ has the avoidance property.
      \item[{\rm (b)}] $\{P_1,\ldots,P_N\}$ is not contained in any hyperplane, and there is no linear set $(L_V,\wt_V)$ of rank $Kh-2$ in $\PG(K-1,q^h)$ such that $\wt_V(P_i)=\dim_{\Fq}(V\cap P_i)\geq h-1$ for each $i\in[N]$.  
 \end{enumerate} 
\end{theorem}

\begin{proof}
$\{\mF_q(P_1),\ldots,\mF_q(P_N)\}$ has the avoidance property if and only if each codimension~$2$ subspace in $\PG(Kh-1,q)$ does not meet all the $(h-1)$-dimensional spaces $\mF_q(P_i)$ in at least an $(h-2)$ space. By first translating this property in $\PG(K-1,q^h)$ and then going to the vectorial setting $\F_{q^h}^K$, this means that there is no codimension $2$ $\Fq$-subspace $V$, meeting  $\langle v_i\rangle_{\F_{q^h}}$ in at least an $(h-1)$-dimensional $\Fq$-subspace. We conclude by observing that, for $h=2$, a codimension $2$ $\Fq$-subspace of $\F_{q^2}^K$ is either an $\F_{q^2}$-hyperplane or a $[2K-2,K]_{q^2/q}$ system, while, for $h>2$, a codimension $2$  $\Fq$-subspace of $\F_{q^2}^K$ can only be a $[2K-2,K]_{q^2/q}$ system. 
\end{proof}

\begin{remark}
Observe that a generalization of \cite[Theorem 3.15]{bartoli2022cutting} has been pointed out by Denaux in \cite[Theorem 16]{denaux2022higgledy}. His result states that if a set of points $\{P_1,\ldots,P_N\}\subseteq \PG(K-1,q^h)$ is not contained in any $\Fq$-linear set, then its field reduction $\mF_q(P_1)\cup\ldots\cup\mF_q(P_N)$ is a \textbf{strong} $(h-1)$-\textbf{blocking set}, that is a set of points whose intersection with every codimension $(h-1)$ space generates the space itself.
 This result is however different from Theorem \ref{thm:characterization_outer_strong}, even though they coincide for $h=2$. We will now study this case in more detail.
\end{remark}

For the special case when $h=2$, we can reformulate Theorem \ref{thm:characterization_outer_strong}, and obtain also a connection with degenerate Hermitian varieties.

We recall that an \textbf{Hermitian variety} in $\PG(K-1,q^2)$ is 
a set of the form
$$\mV_{H}:=\{x \in \PG(K-1,q^2) \St x H \sigma(x)^\top=0\},$$
where $\sigma$ is the (component-wise) $q$-Frobenius automorphism of $\F_{q^2}$, such that $\sigma(\alpha)=\alpha^q$ for every $\alpha \in \F_{q^2}$, and $H\in \F_{q^2}^{K\times K}$ is an Hermitian matrix, that is, $H=\sigma(H)^\top$. The \textbf{rank} of $\mV_{H}$ is the rank of the matrix $H$. If the rank of $H$ is smaller than $K$, the Hermitian variety is said to be \textbf{degenerate}.

Note that, a rank-$1$ degenerate Hermitian variety is just a hyperplane of $\PG(K-1,q^2)$, while a rank $2$ degenerate Hermitian variety is a cone with as vertex a $(K-3)$ dimensional subspace and base a Baer subline. 
In other words, it is $\PGL(K,q^2)$-equivalent to the rank $(2K-2)$ linear set $(L_V,\wt_V)$, with 
$$ V=\{(\alpha_1,\ldots,\alpha_K)\St \alpha_i \in \F_{q^2} \mbox{ for each } i \in [K-2], \, \alpha_{K-1}, \alpha_K\in \Fq\}.$$

\begin{theorem}\label{thm:TFAE}
 Let $\mP=\{P_1,\ldots,P_N\}\subseteq \PG(K-1,q^2)$ be a set of points not all lying on the same hyperplane. The following are equivalent.
 \begin{enumerate}
      \item[{\rm (a)}] $\{\mF_q(P_1),\ldots,\mF_q(P_N)\}$ has the avoidance property.
      \item[{\rm (b)}]  There is no linear set $(L_V,\wt_V)$ of rank $2K-2$ in $\PG(K-1,q^2)$ such that $P_i \in L_V$ for each $i\in [N]$. 
      \item[{\rm (c)}] $P_1,\ldots,P_N$ are not  contained in the same rank $2$ degenerate Hermitian variety in $\PG(K-1,q^2)$.
 \end{enumerate} 
\end{theorem}

\begin{proof}
The equivalence between (a) and (b) follows from Theorem \ref{thm:characterization_outer_strong}, observing that, for any linear set $(L_V,\wt_V)$, $\wt_V(P_i)\geq 1$ if and only if $P_i\in L_V$. Furthermore, 
using a dimension argument, every $[2K-2,K]_{q^2/q}$ system $V'$ is $\GL(K,q^2)$-equivalent to a space of the form
$$\{(\alpha_1,\ldots,\alpha_K)\St \alpha_i \in \F_{q^2} \mbox{ for each } i \in [K-2], (\alpha_{K-1},\alpha_K)\in T\}, $$
for some $[2,2]_{q^2/q}$ system $T$. However, any $[2,2]_{q^2/q}$ system is $\GL(2,q^2)$-equivalent to $\Fq^2$ and thus the linear set associated to $V'$ is equivalent to a rank $2$ degenerate Hermitian variety.
\end{proof}

Note that, for $K=3$, the above result  was indeed proved in \cite{bartoli2022cutting}.

As a byproduct, we derive a lower bound on the minimum number of points not contained in a degenerate rank $r$ variety with $1\leq r \leq 2$. 

\begin{corollary}\,
\begin{enumerate}
    \item The smallest size of an outer strong blocking set in $\PG(K-1,q^2)$ is at most the minimum number of points in $\PG(K-1,q^2)$ which are not contained in a hyperplane nor in a rank $2$ degenerate Hermitian variety. 
    \item Every set of $3K-3$ points in $\PG(K-1,q^2)$ is contained in a degenerate rank $r$ Hermitian variety with $1\leq r\leq 2$.
\end{enumerate}
\end{corollary}

\begin{remark}
    It was shown in \cite{randrianarisoa2020geometric} that $q$-systems correspond to nondegenerate rank-metric codes. An $\Fhk$ \textbf{rank-metric} \textbf{code} is an $\F_{q^h}$-linear subspace $\mC \subseteq \F_{q^h}^n$ of dimension~$k$, endowed with the rank distance, defined as $\mathrm{d}_{\mathrm{rk}}(u,v):=\rk(u-v)$ for every $u,v\in\F_{q^h}^n$, where $\rk(v):=\dim_{\F_q}\langle v_1, \ldots, v_n\rangle _{\F_q}$ for $v = (v_1,\ldots,v_n)\in\F_{q^h}^n$.
 A \textbf{generator matrix} of an $\Fhk$ rank-metric code is a matrix $\smash{G \in \F_{q^h}^{k \times n}}$ whose rows generate $\mC$ as an $\F_{q^h}$-linear space. $\mC$ is said to be \textbf{nondegenerate} if the columns of one (and hence all) generator matrix are linearly independent over $\F_q$. For a vector 
$v \in \F_{q^h}^n$ and an ordered basis $\Gamma=\{\gamma_1,\ldots,\gamma_h\}$ of the field extension $\F_{q^h}/\F_q$, let $\Gamma(v) \in \F_q^{n \times h}$ be the matrix defined by
$$v_i= \sum_{j=1}^h \Gamma(v)_{ij} \gamma_j.$$ The \textbf{support} of $v$ is the column space of $\Gamma(v)$ for any basis $\Gamma$ and we denote it by $\sigma_\rk(v)$. The support of $\mC$, denoted by $\sigma_\rk(C)$ is the subspace sum of the supports of its codewords. The correspondence between rank-metric codes and $q$-systems works as follows. Let $\mC$ be an $\Fhk$ nondegenerate rank-metric code with generator matrix $G$. Then, the $\F_q$-span~$V$ of the columns of $G$ is an $\Fhk$ system. Conversely, if $V\subseteq \F_{q^h}^k$ is an $\Fhk$ system and $G\in\F_{q^h}^{k\times n}$ is a matrix whose columns form a basis of $V$, then, clearly, the rows of $G$ generate an $\Fhk$ rank-metric code. Considering this relation, we can re-interpret Theorems \ref{thm:characterization_outer_strong} and \ref{thm:TFAE} in coding theoretical terms. In particular, items (a) and (b) of Theorem \ref{thm:characterization_outer_strong} are equivalent to 
    \begin{enumerate}
        \item[(c)] There is no generator matrix $G\in \F_{q^h}^{K\times (Kh-2)}$ of a nondegnerate $[Kh-2,K]_{q^h/q}$ rank-metric code such that $\dim_{\Fq}(\sigma_\rk(\mathrm{rowsp}(A_iG)))\leq (K-1)h-1$ for each $i \in [N]$, where $\mathrm{rowsp}(A_i)= v_i^\perp$ and $P_i=\langle v_i\rangle_{\F_{q^h}}$.
    \end{enumerate}
    Items (a), (b) and (c) from Theorem \ref{thm:TFAE} are equivalent to 
    \begin{enumerate}
        \item[(d)]  There is no generator matrix $G\in \F_{q^2}^{K\times (2K-2)}$ of a nondegnerate $[2K-2,K]_{q^2/q}$ rank-metric code such that $\dim_{\Fq}(\sigma_\rk(\mathrm{rowsp}(A_iG)))\leq 2K-3$ for each $i \in [N]$, where $\mathrm{rowsp}(A_i)= v_i ^\perp$ and $P_i=\langle v_i\rangle_{\F_{q^2}}$. 
    \end{enumerate}
\end{remark}

\subsection{Iterative Super Construction} In this section we illustrate an explicit inductive construction of strong blocking sets in $\PG(k-1,q)$ -- and hence $k$-dimensional minimal codes over $\Fq$ -- which relies on the following three ingredients: the construction of a rational normal tangent set, Theorem~\ref{thm:4points} and Lemma \ref{lemma_prop4.5}. The size of the obtained strong blocking sets is  \emph{almost linear} in $k$, in a sense that will be clarified with Theorem \ref{thm:super_construction}. However, compared to the known explicit constructions of families of strong blocking sets whose size is linear in $k$ \cite{bartoli2021small,alon2023strong}, there is a significant gain in the computational cost needed for their construction.

\medskip 

We start as a base step with the only strong blocking set of the projective line $\mB_0=\PG(1,q)$.
The inductive step of this construction is based on the following argument. Suppose that we want to construct a strong blocking set in $\PG(k-1,q)$. Let 
$$j:=\min\Big\{ s \in \mathbb N \St s \mid k, \, \frac{2k}{s}-3 \leq q^s \Big\}.$$
Then, we can construct a strong blocking set $\mathcal B\subseteq \PG(\frac{k}{j}-1,q^j)$ as a rational normal tangent set, that is the union of  $\left(2\frac{k}{j}-3\right)$ lines.  Thus, by Theorem \ref{thm:4points}, we can select four suitable points from each line (that is, not belonging to any proper subline), obtaining an outer strong blocking set in $\PG(\frac{k}{j}-1,q^j)$. Then, using field reduction, we get a strong blocking set in $\PG(k-1,q)$ formed by the union of $(j-1)$-dimensional subspaces, and taking in each of them an isomorphic copy of a strong blocking set in $\PG(j-1,q)$, we obtain a smaller a strong blocking set in $\PG(k-1,q)$, by Lemma \ref{lemma_prop4.5}.
Observe that this last step is equivalent to concatenate the corresponding outer minimal $[4(2\frac{k}{j}-3),\frac{k}{j}]_{q^j}$ code with a minimal $[n,j]_q$ code.

\medskip 

In the following, we explain a concrete construction of strong blocking sets for a special sequence of dimensions.

\begin{construction}\label{constr:super} For odd $q$, we  construct a sequence of strong blocking sets $\mathcal B_i\subseteq \PG(k_i-1,q)$ using the  arguments above, where $\{k_i\}_{i \in \N}$ is defined as
$$\begin{cases}k_0=2, \\ 
k_{i+1}=\dfrac{k_i(q^{k_i}+3)}{2}.\end{cases}$$
We start with $\mB_0=\PG(1,q)$. By induction, from the strong blocking set $\mB_i\subseteq \PG(k_i-1,q)$, we construct $\mB_{i+1} \subseteq \PG(k_{i+1}-1,q)$ as follows. 
\begin{enumerate}
    \item We work in $\PG(\frac{k_{i+1}}{k_i}-1,q^{k_i})=\PG(\frac{q^{k_i}+3}{2}-1,q^{k_i})$. In this space, we take a rational normal tangent set 
    $$\mathcal D_i=\bigcup_{j=1}^{q^{k_i}}\ell_j.$$
    \item For each line $\ell_j$, we select $4$ points $\{\alpha_1^{(j)},\alpha_1^{(j)},\alpha_1^{(j)},\alpha_1^{(j)}\}$ not belonging to the same proper subline. 
    \item The field reduction  $\mF_q(\alpha_t^{(j)})$ of each of the $4q^{k_i}$ points selected before is a $(k_i-1)$-dimensional subspace. In each of these $(k_i-1)$-dimensional space we select an isomorphic copy of $\mB_i$. We define $\mB_{i+1}$ to be their union.
\end{enumerate}
\end{construction}

The following result shows that we are indeed producing a sequence of strong blocking sets, whose size is linear in $q$ and almost linear in $k$. In particular, it depends on the \textbf{iterated logarithm} in base $q$, which is defined over the reals as

$$\log_q^{*}n:=\begin{cases}0&{\mbox{if }}n\leq 1;\\1+\log_q ^{*}(\log_q n)&{\mbox{if }}n>1.
\end{cases}$$

\begin{theorem}\label{thm:super_construction}
 For every $i \in \N$, the set $\mB_i$ given in Construction \ref{constr:super} is a strong blocking set in $\PG(k_i-1,q)$, whose size is upper bounded by
 $$ |\mB_i|\leq \frac{8^{\log_q^*(k_i)}}{2}k_i(q+1).$$
\end{theorem}

\begin{proof}
 We show that $\mB_i$ is a strong blocking set in $\PG(k_i-1,q)$ by induction on $i$. For $i=0$ it is clear, since $\mB_0=\PG(k_0-1,q)=\PG(1,q)$.
 
 Now, assume that $\mB_i$ is a strong blocking set in $\PG(k_i-1,q)$. First of all, the rational normal tangent set $\mD_i$ constructed at step (1) is a strong blocking set in $\PG(\frac{q^{k_i}+3}{2}-1,q^{k_i})$, and it is taken as union of lines. Furthermore, by Theorem \ref{thm:4points}, the set 
 $$\bigcup_{j=1}^{q^{k_i}}\bigcup_{t=1}^4\mF_q(\alpha_t^{(j)})$$
 is a strong blocking set, and it is obtained as union of $(k_i-1)$-dimensional subspaces. Finally, $\mB_{i+1}$ is obtained as the union of an isomorphic copy of a strong blocking set $\mB_i$ in each of the $(k_i-1)$-dimensional subspaces. By Lemma \ref{lemma_prop4.5}, this is a strong blocking set. 
 
 Let $n_i=|\mB_i|$. It is immediate to see that the sequence $\{n_i\}_{i \in \N}$ satisfies the following recurrence relation:
 $$\begin{cases}n_0=q+1, \\
n_{i+1}=4\Big(2\dfrac{k_{i+1}}{k_i}-3\Big)n_i=4{q^{k_i}}n_i.
\end{cases}$$
The inverse of the rate $R_i$ of the corresponding minimal $[n_i,k_i]_q$ code is therefore given by
$$\begin{cases}R_0^{-1}=\dfrac{n_0}{k_0}=\dfrac{q+1}{2}, \\
R_{i+1}^{-1}=\dfrac{n_{i+1}}{k_{i+1}}=\dfrac{8{q^{k_i}}n_i}{k_i(q^{k_i}+3)}.
\end{cases}$$
We can immediately see that $R_{i+1}^{-1}\leq 8R_i^{-1}$. Thus, the sequence of inverses of rates is upper bounded by
$$R_{i}^{-1}\leq \dfrac{q+1}{2}8^i.$$
and the size of $\mathcal B_i$ is upper bounded by
$$|\mathcal B_i|=n_i={R_i^{-1}}k_i\leq \frac{q+1}{2}8^ik_i.$$ 
In the above formula, $i=\Upsilon^{-1}(k_i)$, where we simply put $\Upsilon(i)=k_i$. Note that, since 
$k_{i+1}\geq q^{k_i}$, we have that the map $\Upsilon$ grows at least as fast as the \textbf{tetration} in base $q$, that is defined as
$$^iq:=\underbrace{q^{q^{\cdot ^{\cdot ^{q}}}}}_{i \mbox{\footnotesize{ times}}}.$$
Thus, the inverse of $\Upsilon$ grows at most as fast as the inverse of the tetration, which has the same behavior as the {iterated logarithm}. 
Therefore, we can conclude that our construction of strong blocking sets $\mathcal B\subseteq \PG(k-1,q)$ has size
$$ |\mB|\leq 8^{\log_q^*(k)}k(q+1).$$ 
\end{proof}

\begin{remark}
We can easily adapt Constraction \ref{constr:super} to even $q$ case. We consider the sequence of dimensions $\{k_i\}_{i \in \N}$ as
$$\begin{cases}k_0=2, \\ 
k_{i+1}=\dfrac{k_i(q^{k_i}+2)}{2}.\end{cases}$$
All the other arguments are similar, and we obtain strong blocking sets in $\PG(k-1,q)$ whose size is again at most $8^{\log_q^*(k)}k(q+1)$.
\end{remark}

\begin{remark}
We want to highlight the fact that we can generalize Construction \ref{constr:super} to many more sequences of dimensions $\{k_i\}_{i \in N}$. Let $q$ be odd and let $s>1$ be a positive integer coprime with $q$. Moreover, let $t=\mathrm{ord}_{(\mathbb Z/s\mathbb Z)^*}(q)$ be the multiplicative order of $q$ in $(\mathbb Z/s\mathbb Z)^*$. Let us assume that we can construct a strong blocking set $\mB_0\subseteq \PG(t-1,q)$. Then, using the exact same argument as in Construction \ref{constr:super}, we can iteratively construct a strong blocking set $\mB_i\subseteq \PG(k_i-1,q)$ for the sequence of dimensions $\{k_i\}_{i \in N}$, defined via $$\begin{cases}k_0=t, \\ 
k_{i+1}=\dfrac{k_i(q^{k_i}-1)}{s}.\end{cases}$$
This ensures that for each $i\in \mathbb N$ we have $k_i\equiv 0 \pmod{t}$ and hence $k_{i+1}$ is divisible by $k_i$. The size of the strong blocking set $\mB_i\subseteq \PG(k_i-1,q)$ follows the recursive law
$$ 
|\mB_{i+1}|=4\left(\dfrac{2(q^{k_i}-1)}{s}-3\right)|\mB_{i}|, $$
yielding  
$$  |\mB_i|\leq \frac{8^{\log_q^*(k_i)}}{t}k_i|\mB_0|. $$
\end{remark}

\begin{remark}
Using the method of \emph{diverted tangents} described in \cite[Section 3.2]{fancsali2014lines}, here we concretely illustrate how Construction \ref{constr:super} works. Suppose that we have constructed the $i$-th code $\mathcal{C}_i$ associated to the strong blocking set $\mB_i\subseteq \PG(k_i-1,q)$. Define 
$$ a_i(\lambda)=\Big(1,\lambda,\lambda^2,\ldots,\lambda^{\frac{k_{i+1}}{k_i}-1}\Big), \quad b_i(\lambda)=\Big(0,1,\varphi_i(2)\lambda,\varphi_i(3)\lambda^2,\ldots,\varphi_i\big(\frac{k_{i+1}}{k_i}-1\big)\lambda^{\frac{k_{i+1}}{k_i}-2}\Big),$$
where $\varphi_i:\{0,\ldots,q^{k_i}-1\}\longrightarrow \F_{q^{k_i}}$ is any bijection such that $\varphi_i(0)=0$ and $\varphi_i(1)=1$.
Take now the  projective $[4q^{k_i},\frac{k_{i+1}}{k_i}]_{q^{k_i}}$ system 
$$\mathcal X_i:=\bigcup_{\lambda \in \F_{q^{k_i}}}\{[a_i(\lambda)],[b_i(\lambda)],[a_i(\lambda)+b_i(\lambda)],[a_i(\lambda)+\omega_i b_i(\lambda)] \}, $$
where $\omega_i \in\F_{q^{k_i}}$ is not contained in any proper subfield (e.g. we can take a primitive element of $\F_{q^{k_i}}$), and let $\mR_i$ be the code associated to $\mathcal X_i$. Then,
$\mC_{i+1}=\mC_i \square \mR_i$
is the minimal $[n_{i+1},k_{i+1}]_q$ code associated to the strong blocking set given in Construction \ref{constr:super}. 
\end{remark}

We conclude this section by analyzing the computational cost of Construction \ref{constr:super}.

\begin{proposition}[Computational cost]
 The strong blocking sets in $\PG(k-1,q)$ given in Construction \ref{constr:super} can be constructed with 
 $$ \mathcal O(k\log^\omega(k)8^{\log_q^*(k)})$$
 operations over $\Fq$, where $\omega\leq 2.371552$ is the matrix multiplication exponent.
\end{proposition}

\begin{proof}
Constructing the strong blocking set $\mB_{i+1}$ described in Construction \ref{constr:super} is equivalent to constructing a generator matrix for the associated code $\mC_{i+1}$. In order to do that, we need to compute the matrix in \eqref{eq:gen_matrix_conc}, where the generator matrix  $G_{\mC}$ of the outer code is obtained by putting the coordinates of the points in $\mathcal X_i$ as columns, and the generator matrix of the inner code is already obtained at the previous step. This is done by first computing the powers of every element $\lambda \in\F_{q^{k_i}}$ up to the $\frac{k_{i+1}}{k_i}$-th power and any bijection $\varphi_i:\{0,1,\ldots,q^{k_i}-1\}\rightarrow \F_{q^{k_i}}$, so that we can compute $a_i(\lambda)$ and $b_i(\lambda)$  for every $\lambda \in \F_{q^{k_i}}$. The cost of doing so is  $\mathcal O(k_{i+1})$ multiplications over $\F_{q^{k_i}}$, that is $\mathcal O(k_{i+1}\log(k_{i+1}))$ operations over $\Fq$. Since the matrix we need to compute is given by \eqref{eq:gen_matrix_conc} and the matrices of the form $A(\alpha)$ are at most $q^{k_i}$, that is the  cardinality of $\F_{q^k_i}$, in order to compute the generator matrix of $\mC_{i+1}$ we only need to perform $q^{k_i}$ matrix multiplications  between a ${k_i}\times {k_i}$ matrix and a ${k_i}\times n_i$ matrix. Since $n_i \in \mathcal O(k_i 8^{\log_q^*{k_i}})$ and $\log^*_q(k_i)\leq
\log^*_q(k_{i+1})-1$, this leads to a  computational cost of 
$$
\mathcal O(k_{i+1}k_i^{\omega}8^{\log^*_q(k_i)})=\mathcal O(k_{i+1}\log^{\omega}(k_{i+1})8^{\log^*_q(k_{i+1})})$$
operations over $\Fq$.
Hence, this is the dominant part of the total cost.
\end{proof}

\begin{remark}
Let us compare Construction \ref{constr:super} with two of the most relevant constructions of small strong blocking sets whose size is linear in $k$.
In \cite{bartoli2021small} a construction of small strong blocking set has been illustrated, whose linear in $k$ and quartic in $q$. Another very recent explicit construction is  based on expander graphs, and its size is linear in both $k$ and $q$; see~\cite{alon2023strong}.

Even though Construction \ref{constr:super} is not linear in $k$ -- and hence does not define a family of asymptotically good minimal codes -- it has a clear advantage in terms of computational cost. Indeed, both the constructions in \cite{bartoli2021small} and in \cite{alon2023strong} rely on the celebrated constructions of asymptotically good algebraic geometry codes beating the Gilbert-Varshamov bound \cite{garcia1996asymptotic,garcia1995tower,MR705893}. The fastest algorithm for constructing such codes can be found in \cite{shum2001low} where it is shown that its cost is  of the order of
$$\mathcal O(k^3\log^3(k)),$$
yielding  a much higher complexity. The price that we have to pay with our construction is in terms of the size, but it is very little, since the iterative logarithm is almost negligible. Indeed, for any practical purpose, $\log^*_q(a)$ can be considered as a constant.
\end{remark}

\medskip

\subsection*{Acknowledgements}

G.~N.~A.\
is supported by the Swiss National Foundation through grant no. 210966. M.~B. is partially supported by the ANR-21-CE39-0009 - BARRACUDA (French \emph{Agence Nationale de la Recherche}). A.~N. is supported by the Research Foundation -- Flanders (FWO) grant number 12ZZB23N.

\bigskip

\bibliographystyle{abbrv}
\bibliography{references.bib}

\end{document}